\pdfoutput=1
\RequirePackage{ifpdf}
\ifpdf 
\documentclass[pdftex]{sigma}
\else
\documentclass{sigma}
\fi

\numberwithin{equation}{section}

\newtheorem{Theorem}{Theorem}[section]
\newtheorem{Lemma}[Theorem]{Lemma}
\newtheorem{Proposition}[Theorem]{Proposition}
 { \theoremstyle{definition}
\newtheorem{Remark}[Theorem]{Remark} }

\def\D{\mathrm{D}}

\def\T{\mathrm{T}}
\def\d{\mathrm{d}}
\def\h{\mathrm{h}}

\def\s{\mathrm{s}}
\def\u{\mathrm{u}}
\def\r{\mathrm{r}}

\def\Cset{\mathbb{C}}

\def\Nset{\mathbb{N}}

\def\Qset{\mathbb{Q}}
\def\Rset{\mathbb{R}}

\def\Zset{\mathbb{Z}}

\def\id{\mathrm{id}}

\DeclareMathOperator{\tr}{tr}
\DeclareMathOperator{\sech}{sech}

\begin{document}
\allowdisplaybreaks

\newcommand{\arXivNumber}{1906.?????}

\renewcommand{\thefootnote}{}

\renewcommand{\PaperNumber}{049}

\FirstPageHeading

\ShortArticleName{Heteroclinic Orbits and Nonintegrability in Hamiltonian Systems}

\ArticleName{Heteroclinic Orbits and Nonintegrability\\ in Two-Degree-of-Freedom Hamiltonian Systems\\ with Saddle-Centers\footnote{This paper is a~contribution to the Special Issue on Algebraic Methods in Dynamical Systems. The full collection is available at \href{https://www.emis.de/journals/SIGMA/AMDS2018.html}{https://www.emis.de/journals/SIGMA/AMDS2018.html}}}

\Author{Kazuyuki YAGASAKI and Shogo YAMANAKA}
\AuthorNameForHeading{K.~Yagasaki and S.~Yamanaka}
\Address{Department of Applied Mathematics and Physics, Graduate School of Informatics,\\ Kyoto University, Yoshida-Honmachi, Sakyo-ku, Kyoto 606-8501, Japan}
\Email{\href{mailto:yagasaki@amp.i.kyoto-u.ac.jp}{yagasaki@amp.i.kyoto-u.ac.jp}, \href{mailto:s.yamanaka@amp.i.kyoto-u.ac.jp}{s.yamanaka@amp.i.kyoto-u.ac.jp}}

\ArticleDates{Received January 29, 2019, in final form June 21, 2019; Published online July 02, 2019}

\Abstract{We consider a class of two-degree-of-freedom Hamiltonian systems with saddle-centers connected by heteroclinic orbits and discuss some relationships between the existence of transverse heteroclinic orbits and nonintegrability. By the Lyapunov center theorem there is a family of periodic orbits near each of the saddle-centers, and the Hessian matrices of the Hamiltonian at the two saddle-centers are assumed to have the same number of positive eigenvalues. We show that if the associated Jacobian matrices have the same pair of purely imaginary eigenvalues, then the stable and unstable manifolds of the periodic orbits intersect transversely on the same Hamiltonian energy surface when sufficient conditions obtained in previous work for real-meromorphic nonintegrability of the Hamiltonian systems hold; if not, then these manifolds intersect transversely on the same energy surface, have quadratic tangencies or do not intersect whether the sufficient conditions hold or not. Our theory is illustrated for a system with quartic single-well potential and some numerical results are given to support the theoretical results.}


\Classification{37J30; 34C28; 37C29}

\renewcommand{\thefootnote}{\arabic{footnote}}
\setcounter{footnote}{0}

\section{Introduction}\label{section1}
Chaotic dynamics and nonintegrability of Hamiltonian systems are classical and fundamental topics in dynamical systems, as seen in the famous work of Poincar\'e \cite{P92}, and they have attracted much attention \cite{K96,MO17,M99,M73,S99}. A~Hamiltonian system is nonintegrable if it exhibits chaotic dynamics (see, e.g.,~\cite{M73}), but the converse is not always true: it may not exhibit chaotic dynamics even if it is nonintegrable. Chaotic dynamics is also very often closely related to the existence of transverse homo- and heteroclnic orbits. For example, if there exist transverse homoclinic orbits to periodic orbits, then a Poincar\'e map appropriately defined is topologically conjugated to a~horseshoe map, which has an invariant set consisting of orbits characterized by the Bernoulli shift, i.e., chaotic dynamics occurs \cite{GH83,M73,W03}. Morales-Ruiz and Peris~\cite{MP99} and Yagasaki~\cite{Y03} discussed a relationship between nonintegrability and chaos for a class of two-degree-of-freedom Hamiltonian systems with saddle centers having homoclinic orbits. They showed that if a~sufficient condition for nonintegrability holds, then there exist transverse homoclnic orbits to periodic orbits. Here we extend their results to a similar class of Hamiltonian systems with saddle centers connected by heteroclinic orbits.

\begin{figure}[t]\centering
\includegraphics[scale=0.8]{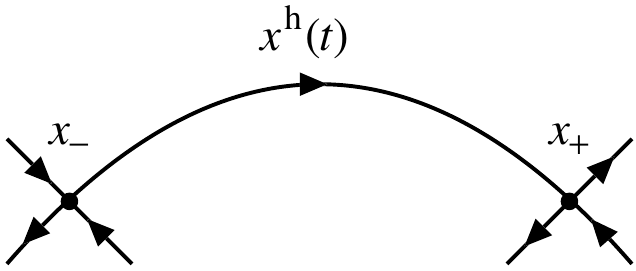}
\caption{Assumptions~(A2) and (A3).}\label{fig:1a}
\end{figure}

More concretely, we consider two-degree-of-freedom Hamiltonian systems of the form
\begin{gather}
\dot{x}=J\D_{x}H(x,y),\qquad
\dot{y}=J\D_{y}H(x,y),\qquad
(x,y)\in\Rset^{2}\times\Rset^{2},\label{eqn:sys}
\end{gather}
where $H \colon \Rset^{2}\times\Rset^{2}\rightarrow\Rset$ is analytic and $J$ represents the $2\times 2$ symplectic matrix,
\[
J=
\begin{pmatrix}
0 & 1\\\
-1 & 0
\end{pmatrix}.
\]
We make the following assumptions.
\begin{itemize}\itemsep=0pt
\item[(A1)] The $x$-plane, $\big\{(x,y)\in\Rset^2\times\Rset^2\,|\,y=0\big\}$, is invariant under the flow of \eqref{eqn:sys}, i.e., $\D_yH(x,0)\allowbreak =0$ for any $x\in\Rset^{2}$.
\item[(A2)] There exist two saddle-centers at $(x,y)=(x_\pm,0)$ on the $x$-plane such that the matrix $J\D^2_xH(x_\pm,0)$ has a pair of real eigenvalues $\lambda_\pm$, $-\lambda_\pm$ and the matrix $J\D_y^2H(x_\pm,0)$ has a~pair of purely imaginary eigenvalues ${\rm i}\omega_\pm$, $-{\rm i}\omega_\pm$ ($\lambda_\pm,\omega_\pm>0$), where the upper and lower signs in the subscripts are taken simultaneously.
\end{itemize}
Assumption (A2) implies that there exist one-parameter families of periodic orbits near the saddle-centers $(x_\pm,0)$ by the Lyapunov center theorem (see, e.g., \cite{MO17}). In addition, the system restricted on the $x$-plane,
\begin{gather}
\dot{x}=J\D_{x}H(x,0),\label{eqn:sysx}
\end{gather}
has saddles at $x=x_\pm$. The reader may think that assumption~(A1) is too restrictive but quite a~few important Hamiltonian systems satisfy this assumption. See, e.g., \cite{SY08,Y02} for such examples.

\begin{itemize}\itemsep=0pt
\item[(A3)] The two saddles $x=x_\pm$ are connected by a heteroclinic orbit $x^\h(t)$ in~\eqref{eqn:sysx}, as shown in Fig.~\ref{fig:1a}.
\end{itemize}
In (A3), if $x_-=x_+$, then $x^\h(t)$ becomes a homoclinic orbit.

In \cite{SY08} a Melnikov-type technique (see, e.g., \cite{GH83,M63} for its original version) was developed for~\eqref{eqn:sys} to detect the existence of transverse heteroclinic orbits connecting periodic orbits near the saddle-centers $(x,y)=(x_\pm,0)$, when $H(x,y)$ is only $C^{r+1}$ ($r\ge 2$). The Melnikov function was defined in terms of a fundamental matrix to the normal variational equation (NVE) along the heteroclinic orbit $(x, y)= \big(x^\h(t),0\big)$,
\begin{gather}
\dot{\eta}=J\D_y^2H\big(x^\h(t),0\big)\eta,\qquad \eta\in\Rset^2,\label{eqn:nveh}
\end{gather}
and such transverse heteroclinic orbits were detected if it has a simple zero. See Section~\ref{section2.1} for more details. This is an extension of a technique developed in~\cite{Y00}, which enables us to show that there exist transverse homoclinic orbits to such periodic orbits and chaotic dynamics occurs~\cite{GH83,W03}, when $x_-=x_+$ and $x^\h(t)$ becomes a homoclinic orbit. Moreover, if there exist transverse heteroclinic orbits from periodic orbits near $(x_-,0)$ to those near $(x_+,0)$ and vice versa, i.e., transverse heteroclinic cycles between the periodic orbits, then so do transverse homoclinic orbits to those near $(x_+,0)$ and $(x_-,0)$, so that the Hamiltonian system~\eqref{eqn:sys} exhibits chaotic dynamics and is nonintegrable. We also point out that Grotta Ragazzo \cite{G94} obtained a~concrete sufficient condition for the occurrence of chaotic dynamics in a special class of~\eqref{eqn:sys} with $x_-=x_+$, using a general result of~\cite{L91}, a little earlier.

On the other hand, Morales-Ruiz and Ramis \cite{MR01} presented a sufficient condition for meromorphic nonintegrability of general complex Hamiltonian systems. Their theory, which is now called the Morales-Ramis theory, states that complex Hamiltonian systems are meromorphically nonintegrable if the identity components of the differential Galois groups \cite{CH11, VS03} for their variational equations (VEs) or NVEs around particular nonconstant solutions such as periodic, homoclinic and heteroclinic orbits are not commutative. See also~\cite{M99}. Ayoul and Zung~\cite{AZ10} used a simple trick called the \emph{cotangent lifting} to show that the Morales-Ramis theory is also valid for detection of meromorphic nonintegrability of non-Hamiltonian systems in the meaning of Bogoyavlenskij~\cite{B98}. Moreover, Morales-Ruiz and Peris \cite{MP99} studied a special class of~\eqref{eqn:sys} with $x_-=x_+$ and showed that
if the Hamiltonian system~\eqref{eqn:sys} is determined by the Morales-Ramis theory to be real-meromorphically nonintegrable,  then chaotic dynamics occurs, using the results of~\cite{G94}. See also~\cite{M99}. Their result was extended to~\eqref{eqn:sys} with $x _− = x_ +$ in~\cite{Y03}, based on the result of~\cite{Y00}. Recently, a further extension on sufficient conditions for real-meromorphic nonintegrability to general dynamical systems having homo- or heteroclinic orbits was accomplished in \cite{YY17}. See Section~\ref{section2.2} for more details.

In this paper, based on \cite{SY08,YY17}, we extend the results of \cite{MP99,Y03} and show the following for~\eqref{eqn:sys} under assumptions (A1)--(A3).
\begin{itemize}\itemsep=0pt
\item \looseness=-1 Assume that $\omega_+=\omega_-$. If sufficient conditions obtained in \cite{YY17} for real-meromorphic nonintegrability near the heteroclinic orbit hold, then the stable and unstable manifolds of pe\-rio\-dic orbits on the same Hamiltonian energy surface near the saddle-centers $(x_\pm,0)$ intersect transversely, i.e., there exist transverse heteroclinic orbits connecting the periodic orbits.
\item Assume that $\omega_+\neq\omega_-$. Then these manifolds intersect transversely, have quadratic tangencies or do not intersect whether the sufficient conditions hold or not. Moreover, under an additional condition, if the sufficient condition does not hold, i.e., a necessary condition for real-meromorphic integrability holds, then these manifolds do not intersect. This may be surprising for the reader since they do not coincide even if the Hamiltonian systems are integrable.
\end{itemize}
Here the associated Hessian matrices of the Hamiltonian are assumed to have the same number of positive eigenvalues: otherwise there exist no periodic orbits near $(x_\pm,0)$ on the same energy surface, as shown in Proposition \ref{prop:3a} below. Our theory is illustrated for a system with quartic single-well potential and some numerical results by using the computer software {\tt AUTO} \cite{DO12} are given to support the theoretical results.

The above results are remarkable since a relationship between the existence of transverse heteroclinic orbits and nonintegrability for Hamiltonian systems, both of which are important properties of dynamical systems, is addressed for the first time, to the authors' knowledge. If not only transverse heteroclinic orbits but also heteroclinic cycles exist, then chaotic dynamics occurs (see the last paragraph of Section~\ref{section2.1}), so that the Hamiltonian systems are nonintegrable. However, if transverse heteroclinic orbits exist but heteroclinic cycles are not formed, then chaotic dynamics may not occur and it is not clear that the systems are nonintegrable. See, e.g., an example in \cite[Section~1.1.2]{Z88}. We remark that in different settings the non-existence of first integrals when transverse heteroclinic orbits to hyperbolic periodic orbits exist was discussed in \cite{D06,Z88}. Moreover, transverse heteroclinic orbits may not exist even if the systems are nonintegrable. Thus, our problem is more subtle, so that our conclusions are more complicated as stated above, compared with the previous one discussed for homoclinic orbits in~\cite{MP99,Y03}.

The outline of this paper is as follows. In Section~\ref{section2} we briefly review the previous results of \cite{SY08} and~\cite{YY17} on the existence of transverse heteroclinic orbits to periodic orbits near $(x_\pm,0)$ and on necessary conditions for real-meromorphic integrability, i.e., sufficient conditions for real-meromorphic nonintegrability. We state the main theorems and prove them in Section~\ref{section3}, and give the example stated above along with numerical results in Section~\ref{section4}.

\section{Previous results}\label{section2}
\subsection{Melnikov-type technique}\label{section2.1}

We first review the result of \cite{SY08} for the existence of transverse heteroclinic orbits in~\eqref{eqn:sys}.

\begin{figure}\centering
\includegraphics[scale=0.85]{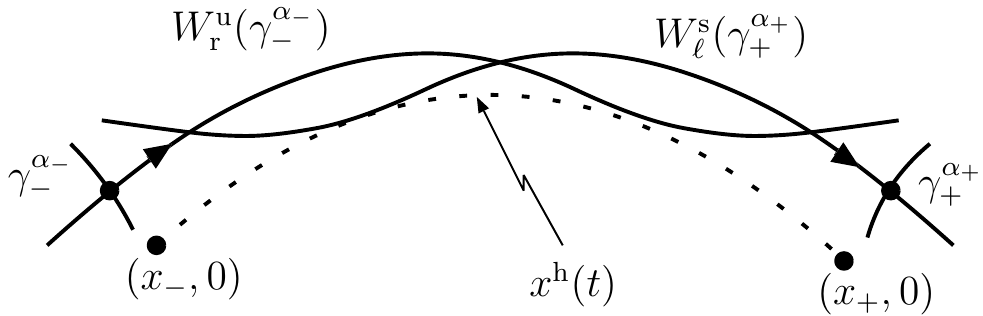}
\caption{The right branch of the unstable manifold of $\gamma_-^{\alpha_-}$ and the left branch of the stable manifold of $\gamma_+^{\alpha_+}$, denoted by $W_\r^\u \big(\gamma_-^{\alpha_-}\big)$ and $W_\ell^\s\big(\gamma_+^{\alpha_+}\big)$, on a Poincar\'e section.}\label{fig:2a}
\end{figure}

Suppose that assumptions (A1)--(A3) hold.
As stated in Section~\ref{section1}, near the saddle-centers $(x_\pm,0)$,
 there exist one-parameter families of periodic orbits,
 which are denoted by $\gamma^{\alpha_{\pm}}_{\pm}, \alpha_{\pm} \in (0, \bar{\alpha}_{\pm}]$,
 with $\bar{\alpha}_{\pm}>0$.
As $\alpha_{\pm} \to 0$,
 they approach $(x_\pm, 0)$ and their periods approach $2\pi/\omega_{\pm}$.
Let $W_\r^\u \big(\gamma_-^{\alpha_-}\big)$ (resp.~$W_\ell^\s\big(\gamma_+^{\alpha_+}\big)$)
 denote the right branch of the unstable manifold of $\gamma_-^{\alpha_-}$
 (resp.\ the left branch of the stable manifold of $\gamma_+^{\alpha_+}$)
 near the heteroclinic orbit $\big(x^{\h} (t), 0\big)$.
See Fig.~\ref{fig:2a}.

Let $\Psi(t)$ denote the fundamental matrix
 of the NVE \eqref{eqn:nveh} along $\big(x^\h(t),0\big)$.
Let $\Phi_\pm(t)$ be the fundamental matrices
 of the NVEs around the saddle-centers $(x_\pm,0)$,
\begin{gather}
\dot{\eta}=J\D_y^2H(x_\pm,0)\eta \label{eqn:nvesc},
\end{gather}
with $\Phi_\pm(0)=\id_2$, where $\id_2$ represents the $2\times 2$ identity matrix.
We easily show that the limits
\begin{gather}
B_-=\lim_{t\rightarrow -\infty}\Phi_-(-t)\Psi(t),\qquad B_+=\lim_{t\rightarrow +\infty}\Phi_+(-t)\Psi(t)
\label{eqn:B}
\end{gather}
exist (cf. \cite[Lemma~3.1]{Y00}) and set $B_0=B_+B_-^{-1}$. We define the \emph{Melnikov function} $M(t_0)$ as
\begin{gather}
M(t_0)=m_-(\eta_0)-m_+(B_0\Phi_-(t_0)\eta_0),\label{eqn:M}
\end{gather}
where $\eta_0\in\Rset^2$ with $|\eta_0|=1$ and
\begin{gather}
m_\pm(\eta)=\frac{1}{2}\eta\cdot\D_y^2 H(x_\pm,0)\eta.\label{eqn:m}
\end{gather}
We have the following theorem (see \cite[Appendix~A]{SY08} for the proof).

\begin{Theorem}\label{thm:2a} For some $\alpha_{\pm}\in(0,\bar{\alpha}_\pm]$,
 let $\gamma_\pm^{\alpha_{\pm}}$ be periodic orbits sufficiently close to $(x_\pm,0)$
 on the same energy surface.
Suppose that $M(t_0)$ has a simple zero.
Then the right branch of the unstable manifold $W_\r^\u \big(\gamma_-^{\alpha_-}\big)$
 and the left branch of the stable manifold $W_\ell^\s \big(\gamma_+^{\alpha_+}\big)$
 intersect transversely on the energy surface, i.e.,
 there exist transverse heteroclinic orbits from~$\gamma_-^{\alpha_-}$ to~$\gamma_+^{\alpha_+}$.
\end{Theorem}

\begin{Remark}\label{rmk:2a}Theorem~\ref{thm:2a} is also valid when $x_+=x_-$. In this situation, if $M(t_0)$ has a simple zero, then the stable and unstable manifolds of periodic orbits near the corresponding saddle-center intersect transversely on the energy surface, i.e., there exist transverse homoclinic orbits to the periodic orbits and consequently chaotic dynamics occurs (e.g., \cite{GH83,W03}). See also~\cite{Y00}.
\end{Remark}

Suppose that there also exists a heteroclinic orbit $\hat{x}^\h(t)$
 from $x_+$ to $x_-$ on the $x$-plane
 and that the hypothesis of Theorem~\ref{thm:2a} holds for both of $x^\h(t)$ and $\hat{x}^\h(t)$.
Then the unstable manifolds of~$\gamma_\mp^{\alpha_\mp}$ intersect
 the stable manifolds of~$\gamma_\pm^{\alpha_\pm}$ transversely on the energy surface
 and these manifolds form a heteroclinic cycle.
This implies that there exist transverse homoclinic orbits to $\gamma_\pm$ (see, e.g., \cite[Section~26.1]{W03}),
 so that chaotic dynamics occurs in~\eqref{eqn:sys}.

\subsection{Necessary conditions for integrability}\label{section2.2}
We next briefly describe the result of \cite{YY17} for integrability of \eqref{eqn:sys} in our setting.

\begin{figure}\centering
\includegraphics[scale=0.8]{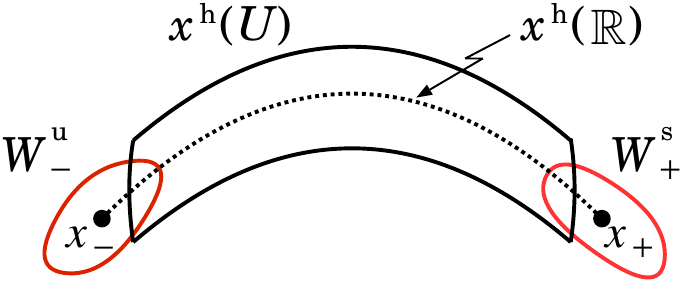}
\caption{Riemann surface $\Gamma=x^\h(U)\cup W_+^\s\cup W_-^\u$.}\label{fig:2b}
\end{figure}

Suppose that (A1)--(A3) hold.
Let $\Gamma_\Rset=\big\{\big(x^\h(t),0\big)\in\Rset^2\times\Rset^2\,|\, t\in\Rset\big\}\cup\{(x_\pm,0)\}$.
Consider the complexification of~\eqref{eqn:sys}
 in a neighborhood of $\Gamma_\Rset$ in $\Cset^4$.
Let $W_\pm^{\s,\u}$ be the one-dimensional local holomorphic stable and unstable manifolds
 of $(x_\pm,0)$ on the $x$-plane.
See \cite{IY08} for the existence of such holomorphic stable and unstable manifolds.
Let $R > 0$ be sufficiently large and let $U$ be a neighborhood of the open interval $(-R,R) \subset \Rset$ in $\Cset$
 such that $x^\h(U)$ contains no equilibrium
 and intersects both $W_+^\s$ and $W_-^\u$.
Here for simplicity we have identified $x^\h(U)\subset\Cset^2$ with $x^\h(U)\times\{0\}$ in $\Cset^4$.
Obviously, $x^\h(U)$ is a one-dimensional complex manifold with boundary.
We take $\Gamma=x^\h(U)\cup W_+^\s\cup W_-^\u$
 and the inclusion map as immersion $i\colon \Gamma\to\Cset^4$.
See Fig.~\ref{fig:2b}.
If $x_+=x_-$ and $x^\h(t)$ is a homoclinic orbit,
 then small modifications are needed in the definitions of $\Gamma$ and $i$.
Let $0_\pm\in\Gamma$ denote points corresponding to the equilibria $x_\pm$.
Taking three charts, $W_\pm^{\s,\u}$ and $x^\h(U)$,
 we rewrite the NVE \eqref{eqn:nveh} along $\Gamma$ as follows
(see \cite[Section~4]{YY17} for the details).

In $x^\h(U)$ we use the complex variable $t\in U$ as the coordinate
 and rewrite the NVE~\eqref{eqn:nveh} as
\begin{gather}
\frac{\d\eta}{\d t}=J\D_y^2H(i(t))\eta,\label{eqn:nve1}
\end{gather}
which has no singularity there.
In $W_+^\s$ and $W_-^\u$
 there exist local coordinates $s_+$ and $s_-$, respectively,
 such that $s_\pm(0_\pm)=0$ and $\d/\d t=h_\pm(s_\pm)\d/\d s_\pm$,
 where $h_\pm(s_\pm)=\mp\lambda_\pm s_\pm+O\big(|s_\pm|^2\big)$ are holomorphic functions.
We use the coordinates $s_\pm$ and rewrite the NVE \eqref{eqn:nveh} as
\begin{gather}
\frac{\d\eta}{\d s_\pm}= \frac{1}{h_\pm(s_\pm)}J\D_y^2H(i(s_\pm))\eta,\label{eqn:nve2}
\end{gather}
which have regular singularities at $s_\pm=0$. Let $M_\pm$ be monodromy matrices of the NVE along~$\Gamma$ around $s_\pm=0$.

Let $\lambda_+'=-\lambda_+$ and $\lambda_-'=\lambda_-$, and let $\mu_\pm=\pm {\rm i}\omega_\pm$ and $\nu_\pm=\mp {\rm i}\omega_\pm$ be eigenvalues of $J\D_y^2H(x_\pm,0)$. Then we have
\[
\frac{\mu_\pm-\nu_\pm}{\lambda_\pm'}
 =\mp\frac{2{\rm i}\omega_\pm}{\lambda_\pm}\not\in\Qset,\qquad
\frac{\mu_\pm+\nu_\pm}{\lambda_\pm'}=0\in\Zset,
\]
which mean that conditions~(A3) and (A4) of \cite{YY17} hold.
Applying Theorem~5.2 of \cite{YY17}, we obtain the following result.

\begin{Theorem}\label{thm:2b}Suppose that assumptions~{\rm(A1)}--{\rm(A3)} hold and the Hamiltonian system \eqref{eqn:sys} is real-meromorphically integrable near $\Gamma_\Rset$. Then the monodromy matrices $M_{\pm}$ are commutative. Moreover, if
\begin{gather}
\frac{\mu_+}{\lambda_+'} - \frac{\mu_-}{\lambda_-'} =-\frac{{\rm i}\omega_+}{\lambda_+}+\frac{{\rm i}\omega_-}{\lambda_-} = 0,\label{eqn:a5}
\end{gather}
then
\begin{gather}
M_+=M_-^{-1}\qquad\mbox{or}\qquad M_+ = M_-.\label{eqn:thm}
\end{gather}
\end{Theorem}

\begin{Remark}\label{rmk:2b}\quad
\begin{enumerate}\itemsep=0pt
\item[\rm(i)]
Let $U_\Rset$ and $U_\Cset$ be, respectively,
 neighborhoods of $\Gamma_\Rset$ in $\Rset^4$ and in $\Cset^4$.
By real-meromorphic integrablity
 we mean that the real Hamiltonian system \eqref{eqn:sys} has an additional first integral
 which is a restriction of some meromorphic function defined in $U_\Cset$ onto $U_\Rset$.
If the Hamiltonian system \eqref{eqn:sys} is real-meromorphically integrable in $U_\Rset$,
 then its complexification is also meromorphically integrable in $U_\Cset$.
Such real-meromorphically nonintegrable Hamiltonian systems
 were also discussed by using a different approach in \cite{MP04,MP05,Z97}.
\item[\rm(ii)]\looseness=-1
Under the hypothesis of Theorem~\ref{thm:2b},
 the identity component $G^0$ of the differential Galois group for the NVE \eqref{eqn:nveh}
 along $\Gamma$ is commutative if and only if so is $M_{\pm}$.
Moreover, if condition~\eqref{eqn:a5} holds,
 then condition~\eqref{eqn:thm} is necessary and sufficient for $G^0$ to be commutative.
\item[\rm(iii)]
If $x_+=x_-$, then condition~\eqref{eqn:a5} automatically holds,
 so that conclusion \eqref{eqn:thm} is necessary for the real-meromorphic integrability of \eqref{eqn:sys}.
We also note that
 the latter case in \eqref{eqn:thm} was overlooked in the early results of {\rm\cite{MP99,Y03}}.
\end{enumerate}
\end{Remark}

\section{Main results}\label{section3}
Let $\sigma_1^{\pm}$ and $\sigma_2^{\pm}$ be eigenvalues of $\D_y^2 H(x_{\pm},0)$.
We have $\sigma_1^{\pm} \sigma_2^{\pm} = \omega_{\pm}^2$,
 so that $\sigma_1^{\pm}$ and $\sigma_2^{\pm}$ are of the same sign,
 where the upper and lower signs in super- and subscripts are taken simultaneously.
Recall that there are one-parameter families of periodic orbits $\gamma_\pm^{\alpha_\pm}$
 near the saddle-centers $(x_\pm,0)$, as stated in Section~2.1.

\begin{Proposition}\label{prop:3a}If $\sigma_1^\pm$ have the opposite signs,
 then there does not exist a pair $(\alpha_+,\alpha_-)$ with $0<\alpha_\pm\ll 1$
 such that the periodic orbits $\gamma_\pm^{\alpha_\pm}$ around $(x_\pm,0)$
 are on the same energy surface.
\end{Proposition}

\begin{proof}Since the saddle-centers $(x_\pm,0)$ are connected by the heteroclinic orbit $\big(x^\h(t),0\big)$,
 we assume that $H(x_+,0)=H(x_-,0)=0$ without loss of generality.
Using the center manifold theorem \cite{GH83,W03},
 we see that there exist center manifolds of $(x_\pm,0)$
 on which $\gamma_\pm^{\alpha_\pm}=\big(x_\pm^{\alpha_\pm}(t),y_\pm^{\alpha_\pm}(t)\big)$ lie.
Moreover, on the center manifolds, the relations $x-x_\pm=O\big(|y|^2\big)$ hold near $(x_\pm,0)$. Hence,
\[
H\big(\gamma_\pm^{\alpha_\pm}\big) = \frac{1}{2}y_\pm^{\alpha_\pm}(t) \cdot\D_y^2H(x_\pm,0)y_\pm^{\alpha_\pm}(t) +O\big(|y^{\alpha_\pm}(t)|^3\big),
\]
which implies that for $\alpha_\pm>0$ sufficiently small
 there does not exist a pair $(\alpha_+,\alpha_-)$ with $H\big(\gamma_+^{\alpha_+}\big)=H\big(\gamma_-^{\alpha_-}\big)$
 if $\sigma_1^+$ and $\sigma_1^-$ have the opposite signs.
\end{proof}

Henceforth we assume that $\sigma_1^\pm$ have the same sign. From the proof of Proposition~\ref{prop:3a}
 we can take $\alpha_+\in(0,\bar{\alpha}_+)$ for $\alpha_-\in(0,\bar{\alpha}_-)$ sufficiently small
 such that $H\big(\gamma_+^{\alpha_+}\big)=H\big(\gamma_-^{\alpha_-}\big)$, i.e.,
 there exist periodic orbits $\gamma_\pm^{\alpha_\pm}$ near $(x_\pm,0)$ on the same energy surface.
Let $M_\pm$ be the monodromy matrices of the transformed NVE \eqref{eqn:nve1} and \eqref{eqn:nve2} around $s_\pm=0$,
 as defined in Section~\ref{section2.1}. We state our main theorems as follows.

\begin{Theorem}\label{thm:1}Assume that $\sigma_1^\pm$ are of the same sign. Let $\alpha_\pm>0$ be sufficiently small
 and satisfy $H\big(\gamma_+^{\alpha_+}\big)=H\big(\gamma_-^{\alpha_-}\big)$. Then the following hold:
\begin{enumerate}\itemsep=0pt
\item[$(i)$]\looseness=-1
If $\omega_+ = \omega_-$ and the monodromy matrices $M_\pm$ are not commutative,
 then the right branch of the unstable manifold $W_\r^\u \big(\gamma_-^{\alpha_-}\big)$ intersects the left branch of the stable manifold $W_\ell^\s \big(\gamma_+^{\alpha_+}\big)$
 transversely on the energy surface, i.e., transverse heteroclinic orbits from $\gamma_-^{\alpha_-}$ to $\gamma_+^{\alpha_+}$ exist.
\item[$(ii)$] If $\omega_+ \neq \omega_-$, then $W_\r^\u \big(\gamma_-^{\alpha_-}\big)$ and $W_\ell^\s \big(\gamma_+^{\alpha_+}\big)$
 intersect transversely on the energy surface, have quadric tangencies or do not intersect.
In particular, they do not coincide.
\end{enumerate}
\end{Theorem}

\begin{Theorem}\label{thm:2}
Assume that $\sigma_1^\pm$ are of the same sign and $\omega_+/\lambda_+ = \omega_- / \lambda_-$.
Let $\alpha_\pm>0$ be sufficiently small
 and satisfy $H\big(\gamma_+^{\alpha_+}\big)=H\big(\gamma_-^{\alpha_-}\big)$.
Then the following hold:
\begin{enumerate}\itemsep=0pt
\item[$(i)$]
If $\omega_+ = \omega_-$ and $M_+\neq M_-^{-1}$, then $W_\r^\u \big(\gamma_-^{\alpha_-}\big)$ intersects $W_\ell^\s \big(\gamma_+^{\alpha_+}\big)$ transversely
 on the energy surface.
\item[$(ii)$]
If $\omega_+ \neq \omega_-$ and $M_+=M_-^{-1}$, then $W_\r^\u \big(\gamma_-^{\alpha_-}\big)$ does not intersect $W_\ell^\s \big(\gamma_+^{\alpha_+}\big)$.
\end{enumerate}
\end{Theorem}

\begin{Remark}\quad
\begin{enumerate}\itemsep=0pt
\item[\rm{(i)}]
The hypothesis of Theorem~{\rm\ref{thm:2}(i)} does not coincide
 with the sufficient condition given in Theorem~\ref{thm:2b}
 for real-meromorphic nonintegrability
 while the hypothesis of Theorem~\ref{thm:1}(i) does.
Similarly, the hypothesis of Theorem~\ref{thm:2}(ii) does not coincide
 with the necessary condition for real-meromorphic integrability.
 \item[\rm{(ii)}]
Assume that $x_-=x_+$ and $x^\h(t)$ is a homoclinic orbit.
Then $\omega_+ = \omega_-$ and $\lambda_+=\lambda_-$.
Hence, we apply Theorem~\ref{thm:2}(i) to recover the result of~\cite{Y03}
 with a necessary correction stated in Remark~\ref{rmk:2b}(iii):
 If $M_+\neq M_-^{-1}$, then the stable and unstable manifolds intersect transversely on the energy surface.
In particular, by Theorem~\ref{thm:2b} and Remark~\ref{rmk:2b}(iii),
 we see that under the sufficient condition for real-meromorphical nonintegrability, the same conclusion holds.
 \end{enumerate}
\end{Remark}

In the rest of this section we prove the main theorems.
We first provide some necessary properties of the Melnikov function $M(t_0)$.
Using \eqref{eqn:m}, we can rewrite \eqref{eqn:M} as
\begin{gather}
M(t_0)=\frac{1}{2}(\Phi_-(t_0)\eta_0)^\T\big( \D_y^2 H(x_-,0)-B_0^\T \D_y^2 H(x_+,0) B_0\big) (\Phi_-(t_0) \eta_0), \label{eq:mel}
\end{gather}
where the superscript {\scriptsize T} represents the transpose operator.
Since the matrix $\D_y^2 H(x_{\pm},0)$ is symmetric,
 there exist a pair of orthogonal matrices $P_{\pm}$ such that
\begin{gather}
P_{\pm}^\T\D_y^2 H(x_{\pm},0) P_{\pm} =
\begin{pmatrix}
\sigma_1^{\pm} & 0\\
0 & \sigma_2^{\pm}
\end{pmatrix}\label{eqn:diag}
\end{gather}
and $\det P_{\pm} = 1$.
Hence, we have
\begin{align}
M(t_0)& =
\frac{1}{2}\big(P_-^\T\Phi_-(t_0)\eta_0\big)^\T\left[
\begin{pmatrix}
\sigma^{-}_1 & 0 \\
0 & \sigma^{-}_2
\end{pmatrix}
 - \tilde{B}_0^\T
\begin{pmatrix}
\sigma^{+}_1 & 0 \\
0 & \sigma^{+}_2
\end{pmatrix}
\tilde{B}_0\right]
\big(P_-^\T \Phi_-(t_0) \eta_0\big)\notag\\
& = \frac{1}{2} \tilde{\eta}(t_0)^\T R \tilde{\eta}(t_0),
\label{eqn:M2}
\end{align}
where $\tilde{B}_0 = P_+^\T B_0 P_-$, $\tilde{\eta} (t_0) = P_-^\T \Phi_- (t_0) \eta_0$ and
\[
R=
\begin{pmatrix}
\sigma^{-}_1 & 0 \\
0 & \sigma^{-}_2
\end{pmatrix}
 - \tilde{B}_0^\T
\begin{pmatrix}
\sigma^{+}_1 & 0 \\
0 & \sigma^{+}_2
\end{pmatrix}
\tilde{B}_0.
\]
On the other hand, there exist a pair of nonsingular matrices $Q_\pm$ such that
\[
Q_\pm^{-1}J\D_y^2H(x_\pm,0)Q_\pm=
\begin{pmatrix}
{\rm i}\omega_\pm & 0\\
0 & -{\rm i}\omega_\pm
\end{pmatrix}.
\]
So we have
\begin{gather}
\Phi_\pm(t)=\exp\big(J\D_y^2H(x_\pm,0)t\big)=Q_\pm
\begin{pmatrix}
{\rm e}^{{\rm i}\omega_\pm t} & 0\\
0 & {\rm e}^{-{\rm i}\omega_\pm t}
\end{pmatrix}
Q_\pm^{-1}.
\label{eqn:Phi}
\end{gather}
Noting that $R$ is symmetric and using~\eqref{eqn:M2} and~\eqref{eqn:Phi}, we immediately obtain the following result.

\begin{Lemma}\label{lem:0}\quad
\begin{enumerate}\itemsep=0pt
\item[$(i)$]
$M(t_0)$ has a simple zero if and only if $\det R < 0$.
\item[$(ii)$]
$M(t_0)$ has no zero if and only if $\det R>0$.
\item[$(iii)$]
$M(t_0)$ is not identically zero but has double zeros
 if and only if $\det R=0$ and $\tr R\neq 0$.
\item[$(iv)$] $M(t_0)$ is identically zero if and only if $\det R=0$ and $\tr R = 0$.
\end{enumerate}
\end{Lemma}

This lemma enables us to easily determine by $\det R$ and $\tr R$ whether $M(t_0)$ is not identically zero or not, whether it has a zero or not, and whether its zero is simple or double if it has.

Denote
\[
\tilde{B}_0 =
\begin{pmatrix}
b_{11} & b_{12} \\
b_{21} & b_{22}
\end{pmatrix}.
\]
Since $\Phi_\pm(t)$ and $\Psi(t)$ are fundamental matrices of linear Hamiltonian systems and $\Phi_\pm(0)=\id_2$ (see Section~\ref{section2.1}), we have $\det B_\pm=\det \Psi(0)$ by~\eqref{eqn:B}, so that
\begin{gather}
\det\tilde{B}_0=\det B_0=1.\label{eqn:detB}
\end{gather}
Hence, we compute
\[
\tr R= -\big(\sigma_1^+ b_{11}^2 + \sigma_1^+ b_{12}^2 + \sigma_2^+ b_{21}^2 + \sigma_2^+ b_{22}^2 \big) + \sigma_1^- + \sigma_2^-
\]
and
\begin{gather}
\det R =
(\omega_+ - \omega_-)^2
 -\left(b_{11} \sqrt{\sigma_1^+ \sigma_2^-} - b_{22} \sqrt{\sigma_2^+ \sigma_1^-}\right)^2
 -\left(b_{12} \sqrt{\sigma_1^+ \sigma_1^-} + b_{21} \sqrt{\sigma_2^+ \sigma_2^-}\right)^2\!
.\!\!\!\label{eqn:det}
\end{gather}
Here we have used the relations $\sigma_1^\pm \sigma_2^\pm = \omega_{\pm}^2$.

\begin{Lemma}\label{lem:1}If $\omega_+ = \omega_-$, then $M(t_0)$ is identically zero or it has a simple zero.
\end{Lemma}

\begin{proof}Assume that $\omega_+ = \omega_-$. Obviously, $\det R \leq 0$ by \eqref{eqn:det}.
If $\det R = 0$, then
\[
b_{11} \sqrt{\sigma_1^+ \sigma_2^-}=b_{22} \sqrt{\sigma_2^+ \sigma_1^-},\qquad
b_{12} \sqrt{\sigma_1^+ \sigma_1^-}=-b_{21} \sqrt{\sigma_2^+ \sigma_2^-},
\]
so that
\begin{align*}
\tr R& =-\sqrt{\sigma_1^+ \sigma_2^+} \left(\sqrt{\frac{\sigma_1^-}{\sigma_2^-}}+\sqrt{\frac{\sigma_2^-}{\sigma_1^-}}\right) (b_{11}b_{22}-b_{12}b_{21})+\sigma_1^-+\sigma_2^-\\
& =-\sqrt{\sigma_1^- \sigma_2^-} \left(\sqrt{\frac{\sigma_1^-}{\sigma_2^-}}+\sqrt{\frac{\sigma_2^-}{\sigma_1^-}}\right) +\sigma_1^-+\sigma_2^-=0.
\end{align*}
Here we have used the relations $\sigma_1^+\sigma_2^+=\sigma_1^-\sigma_2^-$ and $\det\tilde{B}_0=b_{11}b_{22}-b_{12}b_{21}=1$.
Using parts~(i) and~(iv) of Lemma~\ref{lem:0} we obtain the result.
\end{proof}

We also need the following result on the monodromy matrices $M_\pm$ defined in Section~\ref{section2.2}.

\begin{Lemma}\label{lem:2}
The monodromy matrices can be expressed as
\begin{gather}
M_+=B_0^{-1} \exp \left( - \frac{2 \pi {\rm i}}{\lambda_+} J\D_y^2 H(x_+, 0) \right) B_0,\qquad
M_-=\exp \left(\frac{2 \pi {\rm i}}{\lambda_-} J\D_y^2 H(x_-, 0) \right)
\label{eqn:lem2}
\end{gather}
for a common fundamental matrix.
\end{Lemma}
\begin{proof}Let
\[
\tilde{\Psi} (t) = \Psi (t) B_-^{-1}.
\]
Then $\tilde{\Psi}(t)$ is a fundamental matrix of \eqref{eqn:nveh} such that
\[
\lim_{t\to-\infty}\Phi_- (-t)\tilde{\Psi}(t)=\id_2\qquad\mbox{and}\qquad
\lim_{t\to+\infty}\Phi_+(-t)\tilde{\Psi}(t)=B_0.
\]
For the transformed NVE on $\Gamma$,
 we take a fundamental matrix corresponding to $\tilde{\Psi}(t)$.
Since by~\eqref{eqn:Phi} its analytic continuation yields the monodromy matrices
\[
\exp \left(\mp\frac{2 \pi {\rm i}}{\lambda_\pm} J\D_y^2 H(x_\pm, 0) \right)
\]
along small loops around $0_\pm$, we choose the base point near $0_-$ to obtain~\eqref{eqn:lem2}.
\end{proof}

Now we prove the main theorems.

\begin{proof}[Proof of Theorem \ref{thm:1}] Assume that $M(t_0)$ is identically zero. It follows from~\eqref{eq:mel} that
\[
\D_y^2 H(x_-,0) = B_0^\T \D_y^2 H(x_+,0) B_0.
\]
Since $\det B_0 = 1$, we have $B_0 JB_0^\T=J$, so that
\begin{gather}
J\D_y^2 H(x_-,0) = B_0^{-1} J\D_y^2 H(x_+,0) B_0. \label{eq:sim}
\end{gather}
Hence, $J\D_y^2 H(x_-,0)$ and $J\D_y^2 H(x_+,0)$ have the same eigenvalues,
 i.e., $\omega_+ = \omega_-$.
This implies that if $\omega_+\neq\omega_-$,
 then $M(t_0)$ is not identically zero.
Using Lemma~\ref{lem:0} and Theorem~\ref{thm:2a}, we obtain part~(ii).

On the other hand, using Lemma \ref{lem:2} and \eqref{eq:sim}, we see that if $M(t_0)$ is identically zero, then
\begin{gather*}
M_+ =\exp\left(-\frac{2 \pi {\rm i}}{\lambda_+} J\D_y^2 H(x_-, 0)\right),
\end{gather*}
so that $M_\pm$ are commutative. Hence, if $M_\pm$ are not commutative, then $M(t_0)$ is not identically zero. This yields part (i) by Lemma~\ref{lem:1} and Theorem~\ref{thm:2a}.
\end{proof}

\begin{proof}[Proof of Theorem \ref{thm:2}] Assume that $\omega_+ / \lambda_+ =\omega_- / \lambda_-$.
From Lemma \ref{lem:2} and \eqref{eqn:diag} we have
\begin{gather*}
M_+ = B_0^{-1} P_+ \exp\left( - \frac{2 \pi {\rm i}}{\lambda_+}
\begin{pmatrix}
0 & \sigma_2^+ \\
- \sigma_1^+ & 0
\end{pmatrix}
\right)
P_+^{-1} B_0, \\
M_- = P_-
\exp \left( - \frac{2 \pi {\rm i}}{\lambda_-}
\begin{pmatrix}
0 & \sigma_2^- \\
- \sigma_1^- & 0
\end{pmatrix}
\right)
 P_-^{-1}.
\end{gather*}
Using the relations $\sigma_1^\pm \sigma_2^\pm = \omega_\pm^2$, we easily compute
\[
\exp \left( - \frac{2 \pi {\rm i}}{\lambda_{\pm}}
\begin{pmatrix}
0 & \sigma_2^{\pm} \\
- \sigma_1^{\pm} & 0
\end{pmatrix}
\right)
=
\begin{pmatrix}
\cosh 2 \pi \mu & {\rm i}\sqrt{\sigma_2^{\pm}/\sigma_1^{\pm}}\sinh 2 \pi \mu\\
- {\rm i} \sqrt{\sigma_1^{\pm}/\sigma_2^{\pm}}\sinh 2 \pi \mu & \cosh 2 \pi \mu
\end{pmatrix},
\]
where $\mu = \omega_+ / \lambda_+ = \omega_- / \lambda_-$. So the condition $M_+ = M_-^{-1}$ is equivalent to
\begin{gather*}
\tilde{B}_0
\begin{pmatrix}
\cosh 2 \pi \mu & {\rm i}\sqrt{\sigma_2^{-}/\sigma_1^{-}}\sinh 2 \pi \mu\\
- {\rm i}\sqrt{\sigma_1^{-}/\sigma_2^{-}}\sinh 2 \pi \mu & \cosh 2 \pi \mu
\end{pmatrix}\\
\qquad{}=
\begin{pmatrix}
\cosh 2 \pi \mu & {\rm i}\sqrt{\sigma_2^{+}/\sigma_1^{+}}\sinh 2 \pi \mu\\
- {\rm i}\sqrt{\sigma_1^{+}/\sigma_2^{+}}\sinh 2 \pi \mu & \cosh 2 \pi \mu
\end{pmatrix}
\tilde{B}_0,
\end{gather*}
so that
\[
b_{11} \sqrt{\sigma_1^+ \sigma_2^-} - b_{22} \sqrt{\sigma_2^+ \sigma_1^-} = 0, \qquad
b_{21} \sqrt{\sigma_2^+ \sigma_2^-} + b_{12} \sqrt{\sigma_1^{-} \sigma_1^+} = 0.
\]
Hence, if $M_+ = M_-^{-1}$, then by \eqref{eqn:det}
\[
\det R=(\omega_+-\omega_-)^2.
\]
Thus, we obtain part (ii) by Theorem~\ref{thm:2a} and Lemma~\ref{lem:0}.
Moreover, when $\omega_+=\omega_-$,
 the above observation along with~\eqref{eqn:det} shows that
 $\det R=0$ (if and) only if $M_+ = M_-^{-1}$.
This implies part (i) by Theorem~\ref{thm:2a} and Lemma \ref{lem:1}.
\end{proof}

\section{Example}\label{section4}
To illustrate our theory, we consider the two-degree-of-freedom Hamiltonian system
\begin{alignat}{3}
& \dot{x}_1=x_2,\qquad&& \dot{x}_2=-x_1+x_1^3+\tfrac{1}{2}\beta_1 y_1^2+\beta_2 x_1y_1^2,&\nonumber\\
& \dot{y}_1=y_2,\qquad&& \dot{y}_2=-\omega^2 y_1+\beta_1 x_1y_1+\beta_2 x_1^2y_1-y_1^3 &\label{eqn:ex}
\end{alignat}
with the Hamiltonian
\[
H = \tfrac{1}{2} \big(x_2^2 + y_2^2\big) + \tfrac{1}{2} \big(x_1^2 + \omega^2 y_1^2\big)
 - \tfrac{1}{4}\big(x_1^4+y_1^4\big) - \tfrac{1}{2} \beta_1 x_1 y_1^2 - \tfrac{1}{2}\beta_2 x_1^2 y_1^2,
\]
where $\beta_1, \beta_2, \omega \in \Rset $ are constants such that
\begin{gather}
\omega^2 - \beta_2 > |\beta_1|.\label{eqn:excon}
\end{gather}
We easily see that assumption (A1) holds, i.e., the $x$-plane is invariant under the flow of~\eqref{eqn:ex}.
On the $x$-plane, the Hamiltonian system \eqref{eqn:ex} has two saddles at $x=(\pm 1,0)$ with $\lambda_\pm=\sqrt{2}$, and they are connected by a pair of heteroclinic orbits,
\[
x_\pm^\h(t)=\left(\pm\tanh\left(\frac{t}{\sqrt{2}}\right),
\pm\frac{1}{\sqrt{2}}\sech^2\left(\frac{t}{\sqrt{2}}\right)\right),
\]
satisfying
\[
\lim_{t\to+\infty}x_{\pm}^\h(t)=(\pm 1,0)
\qquad\mbox{and}\qquad
\lim_{t\to-\infty}x_{\pm}^\h(t)=(\mp 1,0).
\]
Thus, assumption (A3) holds for $x_\pm=(\pm 1,0)$ or $(\mp 1,0)$,
 where the upper and lower signs are taken simultaneously.
Moreover, by \eqref{eqn:excon},
 the two equilibria in \eqref{eqn:ex} are saddle-centers,
 so that assumption (A2) holds.
In the following,
 we describe the details of computations for $x_\pm=(\pm 1,0)$ and $x_+^\h(t)$,
 from which the corresponding results
 for $x_\pm=(\mp 1,0)$ and $x_-^\h(t)$ also follow immediately.

Let $x_\pm=(\pm 1,0)$. Then
\[
\omega_\pm = \sqrt{\omega^2 \mp \beta_1 - \beta_2},\qquad
\sigma_1^\pm=1,\qquad
\sigma_2^\pm=\omega^2\mp\beta_1-\beta_2>0.
\]
We see that $\omega_+=\omega_-$ if and only if $\beta_1=0$
 and that $\sigma_1^\pm$ are of the same sign.
The NVE \eqref{eqn:nveh} becomes
\begin{gather}
\dot{\eta}_1 = \eta_2, \qquad \dot{\eta}_2 = -\big(\omega^2-\beta_1 x_{1+}^\h (t) - \beta_2 x_{1+}^\h(t)^2\big)\eta_1,
\label{eq:exnve}
\end{gather}
which reduces to the second-order differential equation
\begin{gather}
\ddot{\eta}_1 + \big(\omega^2-\beta_1 x_{1+}^{\h} (t) - \beta_2 x_{1+}^{\h} (t)^2\big)\eta_1 = 0,
\label{eq:exnve1}
\end{gather}
where $x_{1+}^\h(t)$ represents the $x_1$-component of $x_+^\h(t)$, i.e.,
 $x_{1+}^\h(t)=\tanh\big(t/\sqrt{2}\big)$.
Letting $\rho_{\pm} = - {\rm i} \omega_{\pm}/\sqrt{2}$ and using the transformation
\begin{gather}
\tau = \frac{x_{1+}^{\h} (t) + 1}{2}, \qquad \eta_1 = \tau^{\rho_-} (1 - \tau)^{\rho_+} \xi, \label{eq:transformation}
\end{gather}
we rewrite \eqref{eq:exnve1} as the Gauss hypergeometric equation \cite{IKSY,WW27}
\begin{gather}
\tau (1- \tau) \frac{\d^2 \xi}{\d \tau^2} +(c_3 - (c_1 + c_2 + 1) \tau)\frac{\d \xi }{\d \tau} - c_1 c_2 \xi=0, \label{eq:exnve2}
\end{gather}
where
\[
c_1= \chi_+ + \rho_+ + \rho_-,\qquad
c_2 = \chi_- + \rho_+ + \rho_- ,\qquad
c_3 = 2 \rho_- + 1
\]
with $\chi_{\pm} = \frac{1}{2} \big(1 \pm \sqrt {1 + 8 \beta_2}\big)$.
The equilibria $x_-$ and $x_+$ correspond to $\tau=0$ and $1$, respectively.
Singular points of~\eqref{eq:exnve2} are $\tau=0,1,\infty$ and all of them are regular.

The necessary condition for real-meromorphic integrability given by Theorem~\ref{thm:2b} holds only in a limited case for~\eqref{eqn:ex} as follows.

\begin{Lemma}\label{lem:ex}If the monodromy matrices $M_\pm$ are commutative, then
\begin{gather}
\beta_1 = 0,\qquad
\beta_2 = \tfrac{1}{2} n ( n-1)\qquad\text{for some} \quad n\in\Nset \label{eqn:c}
\end{gather}
and $M_+= M_-^{-1}$.
\end{Lemma}

\begin{proof}Let $M_0$ and $M_1$ be the monodromy matrices of~\eqref{eq:exnve2} around $\tau = 0$ and $\tau = 1$, respectively.
Using~\eqref{eq:transformation}, we compute $M_- = e(\rho_-)M_0$ and $M_+ = e(\rho_+)M_1$, where $e(\rho)={\rm e}^{2\pi {\rm i}\rho}$ for $\rho\in\Cset$.
It is a well known fact (see, e.g., \cite[Chapter~2, Theorem~4.7.2]{IKSY}) that the monodromy matrices of~\eqref{eq:exnve2} are given by
\begin{gather*}
M_0 =
\begin{pmatrix}
1 & 0 \\
0 & e(-c_3)
\end{pmatrix},\\
M_1 = \frac{1}{\ell_0}
\begin{pmatrix}
\ell_{11} \ell_{22} - \ell_{12} \ell_{21} e(c_3 - c_1 - c_2) & \ell_{12} \ell_{22} (e(c_3 -c_1 -c_2) -1) \\
\ell_{11} \ell_{21} (1 - e(c_3 - c_1 -c_2)) & \ell_{11} \ell_{22}e(c_3 - c_1 -c_2) - \ell_{12} \ell_{21}
\end{pmatrix},
\end{gather*}
where $\ell_0=\ell_{11} \ell_{22} - \ell_{12}\ell_{21}$,
\begin{alignat*}{3}
&\ell_{11} = \frac{\Gamma (c_3) \Gamma (c_3 -c_1 -c_2)}{\Gamma (c_3 - c_1) \Gamma (c_3 - c_2) },\qquad && \ell_{12} = \frac{\Gamma (2 -c_3) \Gamma (c_3 - c_1 -c_2)}{\Gamma (1-c_1) \Gamma (1 - c_2)},&\\
& \ell_{21} = \frac{\Gamma (c_3) \Gamma (c_1 +c_2-c_3)}{\Gamma (c_1) \Gamma (c_2)}, \qquad && \ell_{22} = \frac{\Gamma (2-c_3) \Gamma (c_1+c_2-c_3)}{\Gamma (c_1 - c_3 + 1) \Gamma ( c_2 - c_3 +1)},&
\end{alignat*}
and $\Gamma(\rho)$ represents the gamma function. Since $c_3 = 2 \rho_- + 1$ and $c_3 - c_1 - c_2 = 1 - \rho_+$ are not integers, we see that if $M_0$ and $M_1$ are commutative, then $M_1$ must be diagonal and consequently $\ell_{12} \ell_{22} = \ell_{11}\ell_{21} = 0$. Moreover, $c_1$ and $c_2$ are not integers, so that $\ell_{12},\ell_{21}\neq 0$, since $1/\Gamma(\rho) = 0$ if and only if $\rho\in \Zset$ and $\rho\le 0$. Hence, if $M_\pm$ are commutative, then $\ell_{11},\ell_{22}=0$.

If $\beta_1 \neq 0$, then $c_3 - c_1$ and $c_3-c_2$ are not integers, so that $\ell_{11},\ell_{22}\neq 0$ and consequently $M_\pm$ are not commutative.
On the other hand, if $\beta_1 = 0$,
 then $c_3 - c_1 = 1 - \chi_+ = c_2 - c_3 + 1$ and $c_3 - c_2 = \chi_+ = c_1 - c_3 + 1$,
 so that $\ell_{11},\ell_{22}=0$ if and only if $\chi_+ \in \Nset$.
Hence, if $M_\pm$ are commutative,
 then $\beta_1 = 0$ and $\chi_+ \in \Nset$, so that the second condition of~\eqref{eqn:c} holds.
Moreover, if condition~\eqref{eqn:c} holds, then $\ell_{11},\ell_{22} = 0$ and $\rho_+ + \rho_- = 0$, so that
\[
M_+ =
\begin{pmatrix}
e(1 - \rho_+) & 0 \\
0 & e(\rho_+)
\end{pmatrix}
=
\begin{pmatrix}
e(\rho_-) & 0 \\
0 & e(1 - \rho_-)
\end{pmatrix}^{-1} = M_-^{-1}.
\]
Thus, we obtain the desired result.
\end{proof}

Obviously, the statement of Lemma~\ref{lem:ex} is also true for $x_\pm=(\mp 1,0)$ and $x_-^\h(t)$.
Let $\gamma_\pm^{\alpha_\pm}$ denote periodic orbits around the saddle-centers at $x=(\pm 1,0)$
 and let $W_\r^\s \big(\gamma_-^{\alpha_-}\big)$ and $W_\ell^\u \big(\gamma_+^{\alpha_+}\big)$
 be the right and left branches of the stable and unstable manifolds
 of $\gamma_-^{\alpha_-}$ and $\gamma_+^{\alpha_+}$, respectively.
Note that $\omega_+/\lambda_+=\omega_-/\lambda_-$ holds if and only if $\beta_1=0$.
Using Theorems~\ref{thm:2b}, \ref{thm:1} and \ref{thm:2} and Lemma~\ref{lem:ex},
 we obtain the following proposition.

\begin{Proposition}\label{prop:ex1}
Suppose that condition~\eqref{eqn:c} does not hold. Then the Hamiltonian system~\eqref{eqn:ex} is real-meromorphically nonintegrable near the heteroclinic orbits $(x,y)=\big(x_\pm^\h(t),0\big)$. Moreover, let $\alpha_\pm>0$ be sufficiently small and satisfy $H\big(\gamma_+^{\alpha_+}\big)=H\big(\gamma_-^{\alpha_-}\big)$. If $\beta_1 = 0$, then $W_\r^\u \big(\gamma_-^{\alpha_-}\big)$ and $W_\ell^\u\big(\gamma_+^{\alpha_+}\big)$, respectively, intersect $W_\ell^\s \big(\gamma_+^{\alpha_+}\big)$ and $W_\r^\s \big(\gamma_-^{\alpha_-}\big)$ transversely on the energy surface, i.e., there exists a heteroclinic cycle. If $\beta_1\neq 0$, then $W_\r^\u \big(\gamma_-^{\alpha_-}\big)$ and $W_\ell^\u\big(\gamma_+^{\alpha_+}\big)$, respectively, intersect $W_\ell^\s \big(\gamma_+^{\alpha_+}\big)$ and $W_\r^\s \big(\gamma_-^{\alpha_-}\big)$ transversely on the energy surface, or these manifolds have quadratic tangencies or do not intersect.
\end{Proposition}

\begin{Remark}\label{rmk:ex1}The existence of such a heteroclinic cycle implies that chaotic dynamics occurs in~\eqref{eqn:ex}, as stated at the end of Section~\ref{section2.1}. From Proposition~\ref{prop:ex1} we immediately see that when $\beta_1\neq 0$, the system \eqref{eqn:ex} is real-meromorphically nonintegrable near the heteroclinic orbits although there may not exist a heteroclinic cycle.
\end{Remark}

We next compute the Melnikov function $M(t_0)$ for \eqref{eqn:ex}. Let $x_\pm=(\pm 1,0)$. The NVE \eqref{eqn:nvesc} becomes
\begin{gather*}
\dot{\eta_1} = \eta_2, \qquad \dot{\eta_2} = -\big(\omega^2\mp \beta_1-\beta_2\big)\eta_1,
\end{gather*}
of which the fundamental matrix with $\Phi_\pm(0)=\id_2$ are given by
\begin{gather}
\Phi_{\pm} (t) =
\begin{pmatrix}
\cos \omega_{\pm} t & \sin \omega_{\pm} t/\omega_{\pm}\\
-\omega_{\pm} \sin \omega_{\pm} t & \cos \omega_{\pm} t
\end{pmatrix}.\label{eqn:exPhi}
\end{gather}
Let $F(c_1,c_2,c_3; \tau)$ be the Gauss hypergeometric function,
\[
F(c_1,c_2,c_3; \tau) =\sum_{k=0}^{\infty} \frac{c_1(c_1+1) \cdots (c_1+k-1) c_2(c_2+1) \cdots (c_2+k-1)}{k! c_3 (c_3+1) \cdots (c_3+k-1)} \tau^k.
\]
Then
\[
\xi=\tau^{1-c_3}F ( c_1 - c_3 + 1, c_2 - c_3 + 1, 2 - c_3; \tau)
\]
is a solutions to \eqref{eq:exnve2} as well as $\xi=F(c_1,c_2,c_3; \tau)$
 (see, e.g., \cite[Chapter~2, Section~1.3]{IKSY} or \cite[Section~14.4]{WW27}).
So we obtain the complex valued solution to~\eqref{eq:exnve1},
\begin{gather*}
\eta=\bar{\eta} (t):=
\left(\frac{1+ \tanh(t/\sqrt{2})}{2}\right)^{-\rho_-} \left(\frac{ 1 - \tanh (t/\sqrt{2})}{2}\right)^{\rho_+}\\
\hphantom{\eta=\bar{\eta} (t):= }{} \times F\left(c_1 - c_3 + 1, c_2 - c_3 + 1, 2 - c_3;\frac{1+ \tanh(t/\sqrt{2})}{2}\right),
\end{gather*}
and the fundamental matrix of \eqref{eq:exnve},
\begin{gather}
\Psi (t)=
\begin{pmatrix}
\operatorname{Re} \bar{\eta} (t) & \operatorname{Im}\bar{\eta} (t)/\omega_- \\
\operatorname{Re} \dot{\bar{\eta}} (t) & \operatorname{Im}\dot{\bar{\eta}} (t)/\omega_-
\end{pmatrix}.\label{eqn:exPsi}
\end{gather}

We easily see that
\[
\left(\frac{1+ \tanh(t/\sqrt{2})}{2} \right)^{-\rho_-} \to 1 \qquad \text{and} \qquad \left(\frac{ 1 - \tanh (t/\sqrt{2})}{2} \right)^{\rho_+} \to {\rm e}^{{\rm i} \omega_+ t}
\]
as $t \to \infty$ and
\[
\left(\frac{1+ \tanh(t/\sqrt{2})}{2}\right)^{-\rho_-} \to {\rm e}^{{\rm i} \omega_- t} \qquad \text{and} \qquad \left(\frac{ 1 - \tanh (t/\sqrt{2})}{2}\right)^{\rho_+} \to 1
\]
as $t \to - \infty$. Thus, we have
\begin{gather}
\bar{\eta} (t) \to {\rm e}^{{\rm i} \omega_- t} \qquad \text{as} \quad t \to -\infty\label{cond:1}
\end{gather}
since
\[
\lim_{\tau \to 0} F ( c_1 - c_3 + 1, c_2 - c_3 + 1, 2 - c_3; \tau) = 1.
\]
Using a well-known formula of the hypergeometric function
 (see, e.g., \cite[Chapter~2, equation~(4.7.9)]{IKSY}), we obtain
\begin{gather*}
\tau^{1-c_3} F(c_1 - c_3 + 1, c_2 - c_3 + 1, 2 - c_3; \tau)= \ell_{12}F(c_1, c_2, c_1+c_2 - c_3 + 1; 1 - \tau)\\
 \qquad{}+ \ell_{22}(1-\tau)^{c_3-c_1-c_2} F ( c_3-c_1, c_3-c_2 , c_3- c_1 - c_2 +1;1- \tau),
\end{gather*}
so that
\begin{gather}
\bar{\eta} (t) \to\ell_{12}{\rm e}^{{\rm i} \omega_+ t} + \ell_{22}{\rm e}^{- {\rm i} \omega_+ t} \qquad \text{as} \quad t \to \infty.\label{cond:2}
\end{gather}
Substituting \eqref{eqn:exPhi} and \eqref{eqn:exPsi} into \eqref{eqn:B} and using \eqref{cond:1} and \eqref{cond:2}, we compute
\[
B_+ =
\begin{pmatrix}
\operatorname{Re} \ell_{12} + \operatorname{Re} \ell_{22} & (\operatorname{Im}\ell_{12} + \operatorname{Im}\ell_{22} )/\omega_-\\
- \omega_+ (\operatorname{Im}\ell_{12} + \operatorname{Im}\ell_{22}) & \omega_+(\operatorname{Re} \ell_{12} - \operatorname{Re} \ell_{22})/\omega_-
\end{pmatrix},\qquad B_- = \id_2,
\]
which yields
\begin{gather}
B_0 = B_-^{-1}B_+ =
\begin{pmatrix}
\operatorname{Re} \ell_{12} + \operatorname{Re} \ell_{22} & (\operatorname{Im}\ell_{12} + \operatorname{Im}\ell_{22} )/\omega_-\\
- \omega_+ (\operatorname{Im}\ell_{12} + \operatorname{Im}\ell_{22}) & \omega_+(\operatorname{Re} \ell_{12} - \operatorname{Re} \ell_{22})/\omega_-
\end{pmatrix}.\label{eq:B}
\end{gather}

Equation~\eqref{eqn:m} becomes
\[
m_\pm(\eta)=\tfrac{1}{2}\big(\big(\omega^2\mp\beta_1-\beta_2\big) \eta_1^2 +\eta_2^2\big).
\]
Using \eqref{eqn:M} and \eqref{eq:B}, we obtain the Melnikov function as
\begin{gather*}
M(t_0) = (- \operatorname{Re} \ell_{12} \operatorname{Re} \ell_{22} + \operatorname{Im}\ell_{12} \operatorname{Im}\ell_{22}) \omega_+^2 \cos 2 \omega_- t_0 \\
\hphantom{M(t_0) = }{} + (\operatorname{Re} \ell_{12} \operatorname{Im}\ell_{22} + \operatorname{Im}\ell_{12} \operatorname{Re} \ell_{22}) \omega_+^2 \sin 2 \omega_- t_0 + \tfrac{1}{2} \big( \omega_-^2 - \big(|\ell_{12}|^2 + |\ell_{22}|^2\big) \omega_+^2\big) \\
\hphantom{M(t_0)}{} = \omega_+^2 |\ell_{12}| |\ell_{22}| \cos (2 \omega_- t_0 - \phi_0) + \tfrac{1}{2} \big(\omega_-^2 - (|\ell_{12}|^2 + |\ell_{22}|^2) \omega_+^2\big) ,
\end{gather*}
where
\[
 \tan \phi_0 = \frac{\operatorname{Re} \ell_{12} \operatorname{Im}\ell_{22} + \operatorname{Im}\ell_{12} \operatorname{Re} \ell_{22}}
 {- \operatorname{Re} \ell_{12} \operatorname{Re} \ell_{22} + \operatorname{Im}\ell_{12} \operatorname{Im}\ell_{22}}.
\]
Let
\begin{gather*}
G( \beta_1, \beta_2 , \omega ) := \big(\omega_+^2 |\ell_{12}| |\ell_{22}|\big)^2
 - \tfrac{1}{4}\big(\omega_-^2 - \big(|\ell_{12}|^2 + |\ell_{22}|^2\big) \omega_+^2\big)^2 \notag \\
\hphantom{G( \beta_1, \beta_2 , \omega )}{} =\omega_+^2 \omega_-^2 |\ell_{22}|^2 - \tfrac{1}{4}\omega_-^2(\omega_+ - \omega_-)^2. 
\end{gather*}
Here we have used the relation $|\ell_{12}|^2-|\ell_{22}|^2=\omega_-/\omega_+$ obtained from \eqref{eqn:detB} and \eqref{eq:B}. The Melnikov function $M(t_0)$ has a simple zero (resp.\ no zero) if and only if $G(\beta_1, \beta_2, \omega) > 0$ (resp.\ $G(\beta_1, \beta_2, \omega) < 0$). Obviously, the above arguments are valid for $x_\pm=(\mp 1,0)$ and $x_-^\h(t)$. Applying Theorem~\ref{thm:2a}, we obtain the following proposition.

\begin{Proposition}\label{prop:ex2}Let $\alpha_\pm>0$ be sufficiently small and satisfy $H\big(\gamma_+^{\alpha_+}\big)=H\big(\gamma_-^{\alpha_-}\big)$. If $G(\beta_1, \beta_2, \allowbreak \omega) > 0$, then $W_\r^\u \big(\gamma_-^{\alpha_-}\big)$ and $W_\ell^\u\big(\gamma_+^{\alpha_+}\big)$, respectively, intersect $W_\ell^\s \big(\gamma_+^{\alpha_+}\big)$ and $W_\r^\s \big(\gamma_-^{\alpha_-}\big)$ transversely on the energy surface, i.e., there exists a heteroclinic cycles. If $G(\beta_1, \beta_2, \omega) < 0$, then $W_\r^\u \big(\gamma_-^{\alpha_-}\big)$ and $W_\ell^\u\big(\gamma_+^{\alpha_+}\big)$, respectively, do not intersect $W_\ell^\s \big(\gamma_+^{\alpha_+}\big)$ and $W_\r^\s \big(\gamma_-^{\alpha_-}\big)$.
\end{Proposition}

\begin{Remark}\quad
\begin{enumerate}\itemsep=0pt
\item[(i)] As expected from Proposition~\ref{prop:ex1}, when $\beta_1=0$, we see that $G(\beta_1, \beta_2, \omega) > 0$
 if and only if the second condition of~\eqref{eqn:c} does not hold, i.e.,
\begin{gather}
\beta_2\neq\tfrac{1}{2} n ( n-1)\qquad\text{for any} \quad n\in\Nset.\label{eqn:rmkex2}
\end{gather}
This follows from the fact that $\ell_{22}\neq 0$ if and only if condition~\eqref{eqn:rmkex2} holds (see the proof of Lemma~\ref{lem:ex}).
\item[(ii)] When $\beta_1\neq0$, Proposition~\ref{prop:ex1} means that the Hamiltonian system~\eqref{eqn:ex} is always real-meromorphically nonintegrable as stated in Remark~\ref{rmk:ex1}, but there may not exist heteroclinic cycles for periodic orbits: the function $G(\beta_1,\beta_2,\omega)$ may be negative.
\end{enumerate}
\end{Remark}

\begin{figure}[th!]\centering
\begin{minipage}[t]{0.46\textwidth}\centering
\includegraphics[scale=0.63]{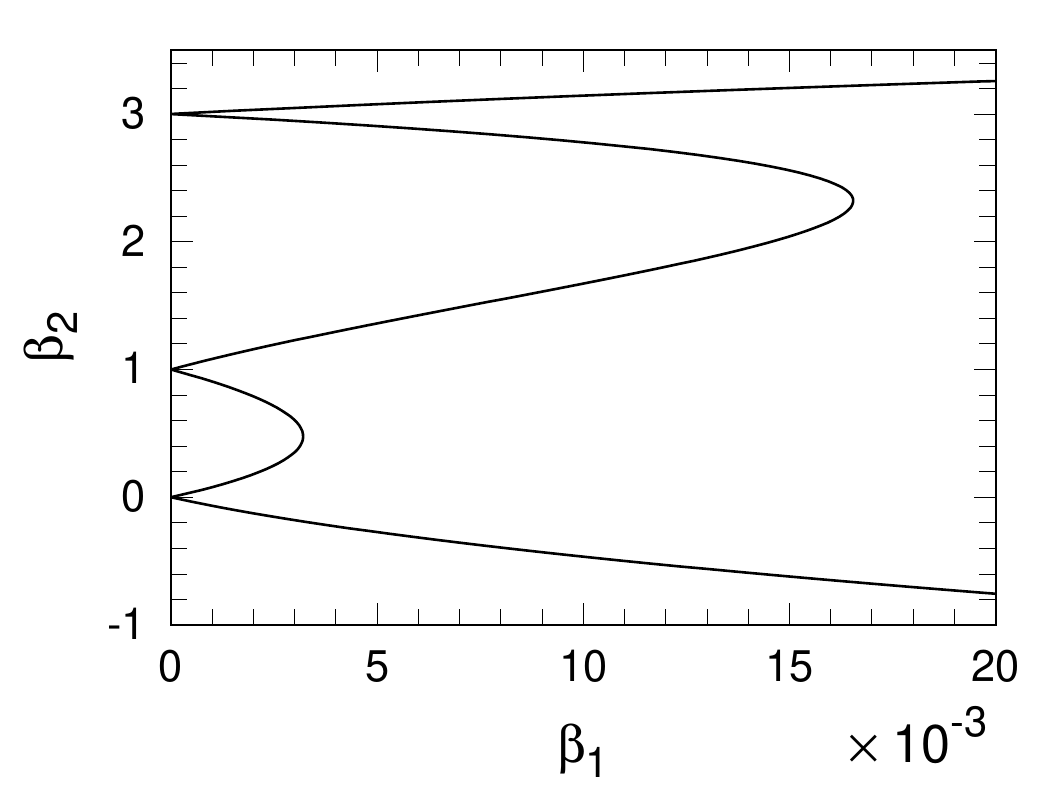}
\caption{Numerical computation of the curve given by $G(\beta_1,\beta_2,\omega) = 0$ with $\omega=2$ in the $(\beta_1,\beta_2)$-plane.}\label{fig:zero}
\end{minipage}
\qquad
\begin{minipage}[t]{0.46\textwidth}\centering
\includegraphics[scale=0.73]{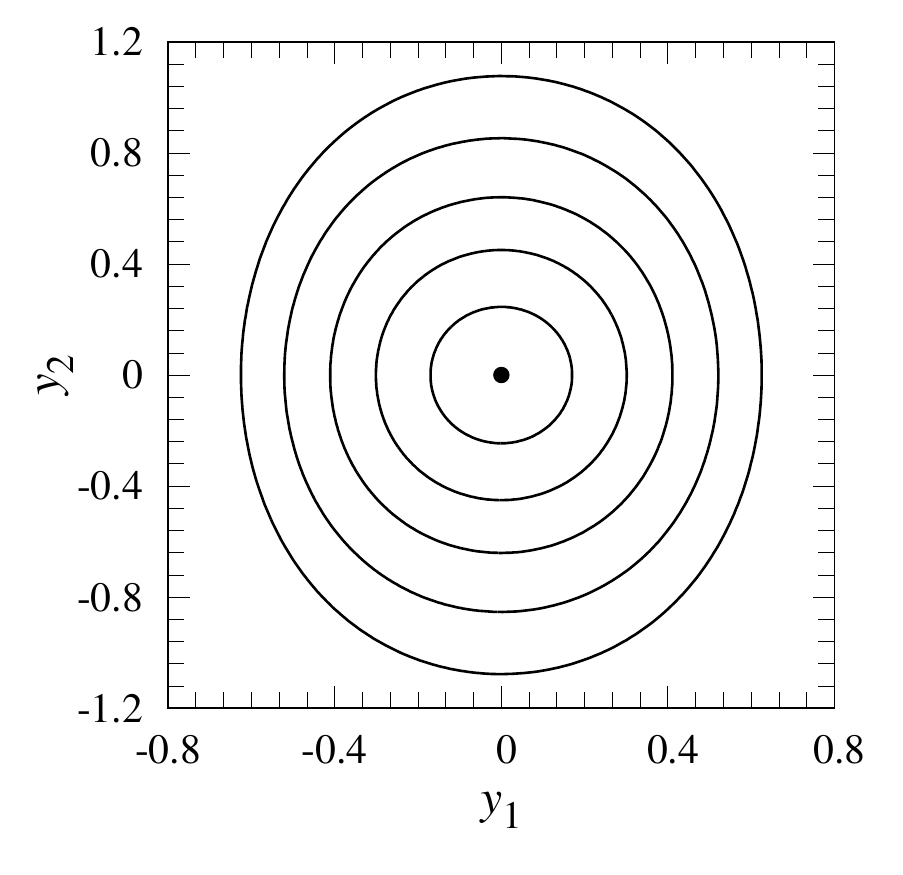}
\caption{Periodic orbits near the saddle-center with $x=(-1,0)$ for $\beta_1 = 5.0\times 10^{-3}$, $\beta_2=2$ and $\omega=2$. Their projections to the $y$-plane are plotted, and their energy values are $H=0.28,0.35,0.45,0.6,0.8$ from the inside.}\label{fig:npo}
\end{minipage}
\end{figure}

\begin{figure}[h!]\centering
\includegraphics[scale=0.72]{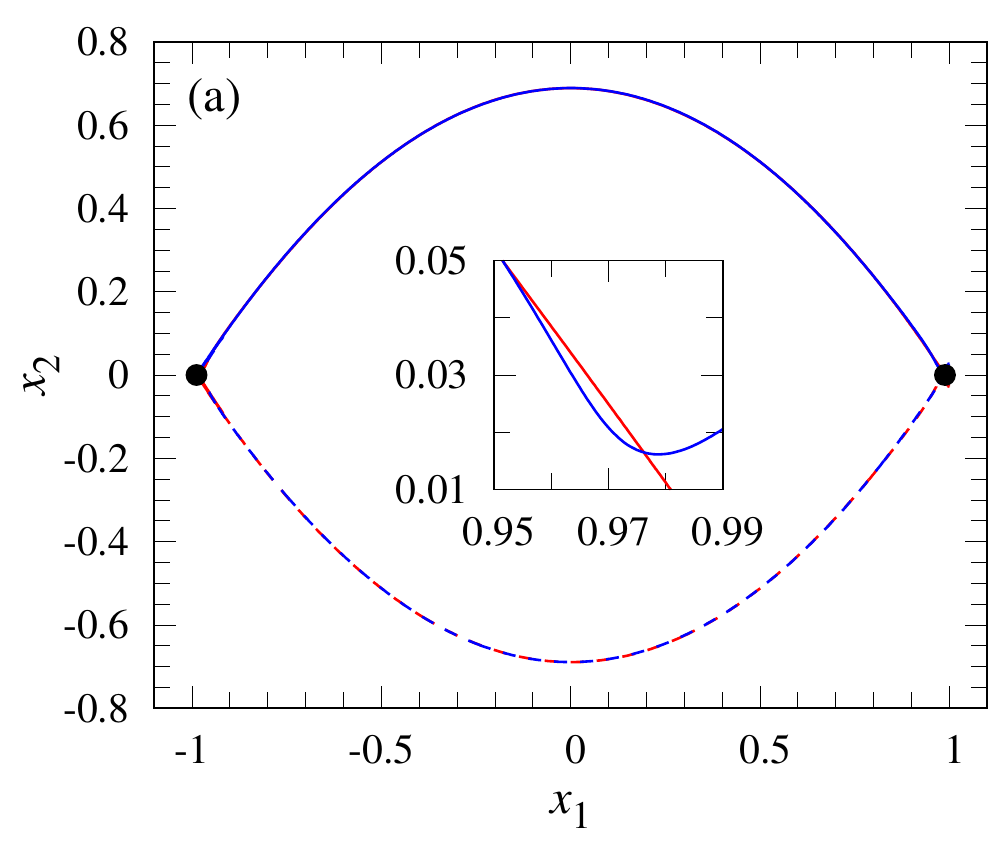}\\[1ex]
\includegraphics[scale=0.72]{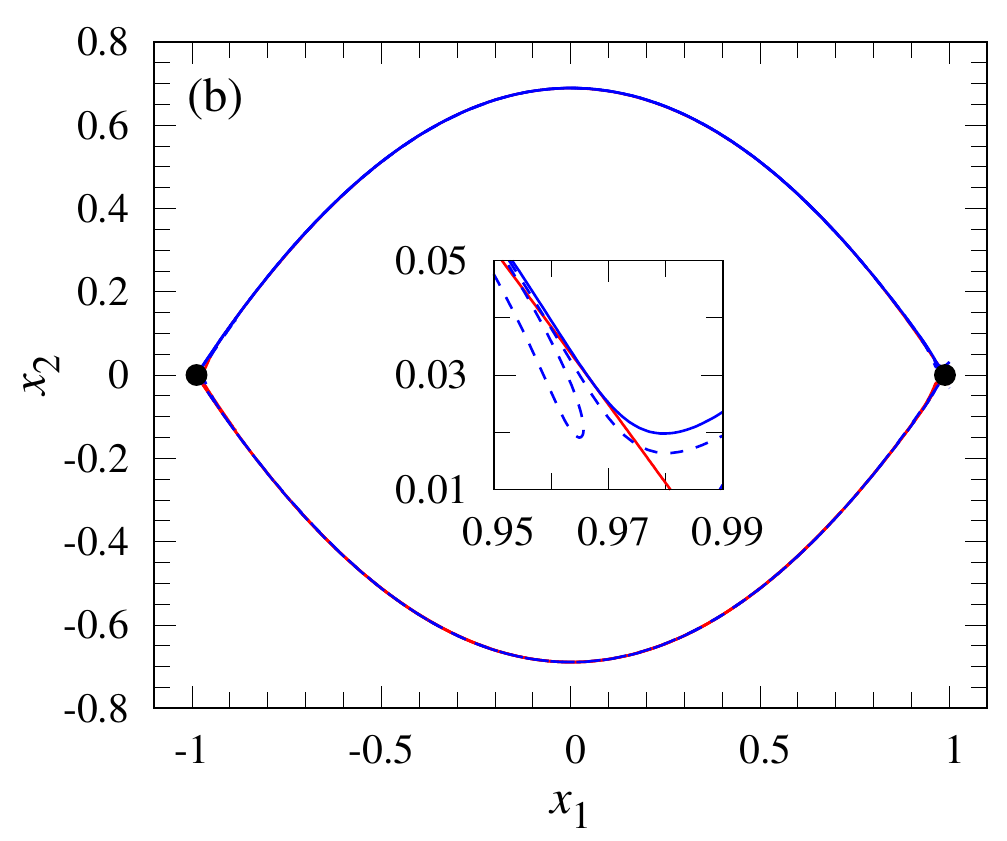}\quad
\includegraphics[scale=0.72]{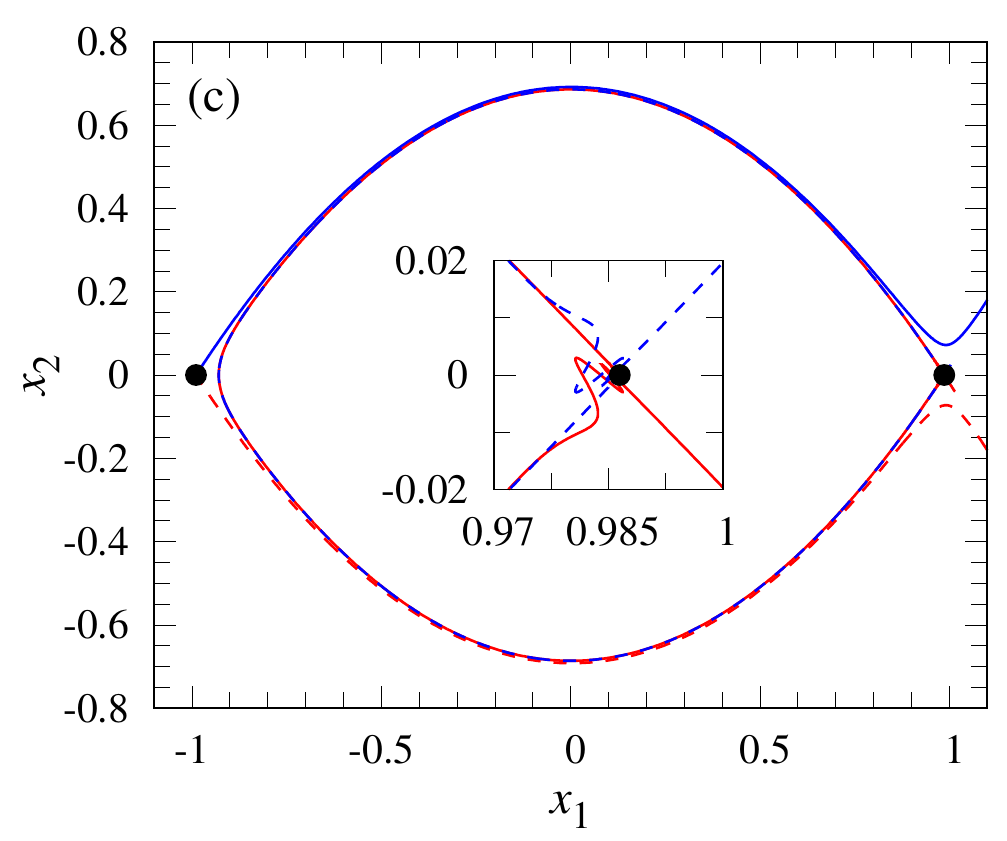}
\caption{Stable and unstable manifolds of periodic orbits near the saddle-centers with $x=(\pm 1,0)$ on the Poincar\'e section $\big\{(x,y) \in \Rset^2 \times \Rset^2 \,|\, y_1 = 0 \big\}$
 for $\beta_2=2$, $\omega=2$ and $H=0.28$: (a)~$\beta_1 = 5\times 10^{-3}$; (b)~$1.32\times 10^{-2}$; (c)~$2\times 10^{-1}$. These manifolds near $\big(x_+^\h(t),0\big)$ and $\big(x_-^\h(t),0\big)$ are plotted as solid and dashed lines, respectively, and blue and red colors are used for the stable and unstable manifolds, respectively.}\label{fig:nsum}
\end{figure}

In Fig.~\ref{fig:zero} we plot the curve given by $G(\beta_1, \beta_2,\omega) = 0$ in the $(\beta_1,\beta_2)$-parameter plane for $\omega=2$. Here we have used the function {\tt fsolve} of {\tt Maple} to numerically solve $G(\beta_1, \beta_2,2) = 0$ for $\beta_2$ varied. By Proposition~\ref{prop:ex2}, heteroclinic cycles on energy surfaces near the saddle-centers exist (resp.\ do not exist) for the parameter values of $\beta_1$, $\beta_2$ in the left (resp.\ right) side of the curve since $G(\beta_1, \beta_2, 2) > 0$ (resp.\ $<0$) there.

To support the above theoretical results, we give numerical computations of the stable and unstable manifolds of periodic orbits near the saddle-centers with $x=(\pm 1,0)$ for the Hamiltonian system~\eqref{eqn:ex}. Our numerical approach was described in \cite[Section~4.3]{SY08} and similar to that of \cite{CL97}. The calculations were carried out by using the numerical computation tool {\tt AUTO}~\cite{DO12}, as in~\cite{CL97,SY08},
although the monodromy matrix (the derivative of the Poincar\'e map) was computed by numerically solving the variational equation around the corresponding periodic orbit directly.

Fig.~\ref{fig:npo} shows numerically computed periodic orbits near the saddle-center with $x=(-1,0)$ for $\beta_1 = 5\times 10^{-3}$, $\beta_2=2$ and $\omega=2$. Similar pictures for periodic orbits were also obtained for the other cases, and periodic orbits far from the saddle-centers could be computed like Fig.~\ref{fig:npo} although the Lyapunov center theorem only guarantees their existence near the saddle-centers.

Fig.~\ref{fig:nsum} shows numerically computed the stable and unstable manifolds, $W^\s\big(\gamma_\pm^{\alpha_\pm}\big)$ and \linebreak $W^\u\big(\gamma_\pm^{\alpha_\pm}\big)$, of periodic orbits $\gamma_\pm^{\alpha_\pm}$ near the saddle-centers on the Poincar\'e section $\big\{(x,y) \in \Rset^2 \times \Rset^2 \,|\, y_1 =0 \big\}$ for $\beta_2=2$, $\omega=2$ and $H=0.28$. In Fig.~\ref{fig:nsum}(a) for $\beta_1=5\times 10^{-3}$, we observe that $W_\r^\u\big(\gamma_-^{\alpha_-}\big)$ and $W_\ell^\u\big(\gamma_+^{\alpha_+}\big)$, respectively, intersect $W_\ell^\s\big(\gamma_+^{\alpha_+}\big)$ and $W_\r^\s\big(\gamma_-^{\alpha_-}\big)$ transversely, and there exists a heteroclinic cycle. In Fig.~\ref{fig:nsum}(b) for $\beta_1=1.32\times 10^{-2}$, $W_\r^\u\big(\gamma_-^{\alpha_-}\big)$ and $W_\ell^\u\big(\gamma_+^{\alpha_+}\big)$, respectively, seem to be quadratically tangent to $W_\ell^\s\big(\gamma_+^{\alpha_+}\big)$ and $W_\r^\s\big(\gamma_-^{\alpha_-}\big)$. In Fig.~\ref{fig:nsum}(c) for $\beta_1=2\times 10^{-1}$, $W_\r^\u\big(\gamma_-^{\alpha_-}\big)$ and $W_\ell^\u\big(\gamma_+^{\alpha_+}\big)$, respectively, do not intersect $W_\ell^\s\big(\gamma_+^{\alpha_+}\big)$ and $W_\r^\s\big(\gamma_-^{\alpha_-}\big)$. We see that for $(\beta_2,\omega)=(2,2)$, $G(\beta_1,\beta_2,\omega)=0$ at $\beta_1\approx 1.5\times 10^{-2}$ in Fig.~\ref{fig:zero}, and predict by Proposition~\ref{prop:ex2} that a heteroclinic cycle exists or not, depending on whether $\beta_1$ is less or greater than the value. Thus, the theoretical prediction fairly agrees with the numerical observation in Fig.~\ref{fig:nsum}. The agreement becomes better when the periodic orbits $\gamma_\pm^{\alpha_\pm}$ are closer to the saddle-centers. In Fig.~\ref{fig:nsum}(c) we also observe that $W^{\s} (\gamma_+^{\alpha_{+}})$ and $W^{\u} \big(\gamma_+^{\alpha_+}\big)$ still intersect transversely. Hence, the Hamiltonian system \eqref{eqn:ex} exhibits chaotic dynamics and it is nonintegrable. This consists with the results of Proposition~\ref{prop:ex1}.

\subsection*{Acknowledgements}
This work was partially supported by Japan Society for the Promotion of Science, Kekenhi Grant Numbers JP17H02859 and JP17J01421. The authors are grateful to Masayuki Asaoka for pointing out the fact stated in Proposition~\ref{prop:3a}. The idea of its proof is also due to him. They also thank the anonymous referees especially for introducing the references \cite{D06,Z88} to them.

\pdfbookmark[1]{References}{ref}
\LastPageEnding

\end{document}